\documentclass{gtpart}

\title[Collatz mapping on $\Z_{10}^{+}$]{Collatz mapping on $\Z/10\Z$}

\author[B. Khanzadeh]{Benyamin Khanzadeh H.}
\address{}
\email{benjaminkhanzade319@gmail.com}



\usepackage{
    amsfonts,
    amsmath,
    amssymb,
    hyperref,
    tikz
}

\usetikzlibrary{arrows}

\newtheorem{lem}{Lemma}[section]

\theoremstyle{definition}
\newtheorem{rem}{Remark}[section]
\newtheorem{con}{Conjecture}

\newcommand{\N}{\mathbb{N}}
\newcommand{\Col}{\operatorname{Col}}
\newcommand{\PP}[1]{{\tt(P#1)}}

\begin{document}

\begin{abstract}
    We introduce the \emph{Collatz conjecture} and its history. Some definition that this conjecture has, will be expressed and with these we try to explain some good lemma to justify the main properties of the \emph{Collatz conjecture}. With these lemmas a situation made to go through to write some mainly properties of Collatz graph on $\mathbb{Z}/10\mathbb{Z}$. Then with something provided it will go to draw the graph on $\mathbb{Z}/10\mathbb{Z}$.
\end{abstract}
\maketitle
\section{Introduction}

The \emph{Collatz conjecture}, also known as the "$3x+1$" problem, is a conjecture in mathematics that concerns sequences of integer numbers. These sequences start with an arbitrary positive integer, and so each term is obtained from the previous one as follows: if the previous term is even, the next term is one half of this one, and if the previous term is odd, the next term is $3$ times the previous one plus $1$. The conjecture says that no matter what is the starting value, the sequence will always reach $1$. \cite{A}

More specifically, consider the \emph{Collatz function} as the map $\Col:\N\to\N$, defined by:\[\Col(x)=\left\{\begin{array}{ll}
\frac{x}{2}&\text{if }x\equiv0\mod2\\
3x+1&\text{if }x\equiv1\mod2
\end{array}\right.\]Thus, the Collatz conjecture says that, for every $x\in\N$, there is a positive integer $k$ such that $\Col^k(x)=1$. \cite{A}

The conjecture is named after Lothar Collatz, who introduced the idea in 1937, two years after receiving his doctorate \cite{C11}. However, it is also known as the \emph{Ulam conjecture} (after Stanisław Ulam), the \emph{Kakutani's problem} (after Shizuo Kakutani), the \emph{Thwaites conjecture} (after Sir Bryan Thwaites), the \emph{Hasse's algorithm} (after Helmut Hasse), or the \emph{Syracuse problem} \cite{C22,C33}.

The sequences of numbers involved here are also referred to as the \emph{hailstone sequences} or the \emph{hailstone numbers},  because the values are usually subject to multiple descents and ascents like hailstones in a cloud \cite{C55} or as wondrous numbers \cite{C77}.

Paul Erdős said about the Collatz conjecture: "mathematics may not be ready for such problems" \cite{C88}. He also offered US\$500 for its solution \cite{C99}. On the other hand, Jeffrey Lagarias stated in 2010 that the Collatz conjecture "is an extraordinarily difficult problem, completely out of reach of present day mathematics" \cite{C99110}.

\section{Main conjecture}

\subsection{Closed loops}

A \emph{closed loop} or simply a \emph{cycle} is a finite sequence $a=(a_0,\ldots,a_k)$ of positive integers, such that $a_k=a_0$ and $a_i=\Col^i(a_{i-1})$ for all $i\in\{1,\ldots,k\}$. Note that every closed loop $a$ as above generates an infinite family of cycles $a^m$ ($m\geq1$), by concatenating $m-1$ copies of $(a_1,\ldots,a_k)$ to the right of $a$, that is\[a^m=(a_0,\underbrace{a_1,\ldots,a_k}_{(m-1)-\text{copies}}).\]Hence, if $a$ is a closed loop as above, then $\Col^{mk}(a_0)=a_0$ for all $m\geq1$.

One of the most important problems about this conjecture is the presence of the closed loop $(1,4,2,1)$, because it creates complexity. Indeed, if $\Col^k(x)=1$, the existence of the mentioned closed loop implies that\[\Col^{3m}(\Col^k(x))=\Col^{3m+k}(x)=1\quad\text{for all}\quad m\geq1.\]

To avoid the contradiction given by the closed loop $(1,4,2,1)$, we will replace the Collatz function by the map $\Col_*:\N\to\N$ defined as follows:\[\Col_*(x)=\left\{\begin{array}{ll}
\frac{x}{2}&\text{if }x\equiv0\mod2\\
3x+1&\text{if }x\equiv1\mod2\text{ and }x>1\\
1&\text{if }x=1
\end{array}\right.\]

For $t\in\N$, we have\[\Col(2t+1)=3(2t+1)+1=6t+4\equiv 4\mod6,\qquad\Col(2t)=\frac{2t}{2}=t.\]Notice that, for $x\in\N$, $\Col^{-1}(\{x\})$ is formed by the element $2x$, and additionally, by the element $\frac{x-1}{3}$ only if $x\equiv4\mod6$.

\section{Classifications and fields}


For every $m\in\Z$, denote by $[m]$ the class of it module ten, i.e.  $\Z/{10}\Z=\{[0],\ldots,[9]\}$. Note that $\N=[0]^+\sqcup\cdots\sqcup[9]^+$, where $[k]^+=[k]\cap\N$, that is:
\[\begin{array}{rcl}
[0]^+&=&\{10,20,\ldots\},\\[0.8mm]
[1]^+&=&\{1,11,21,\ldots\},\\[0.8mm]
&\cdots&\\[0.8mm]
[9]^+&=&\{9,19,29,\ldots\}.
\end{array}\]Hence, $[k]^+=\{10t+k\mid t\geq0\}=[0]^++\{k\}$.

As every natural number can be written as $\sum_{i=1}^na_i10^i$ with $0\leq a_i<10$, then\begin{equation}\label{000}[k]^+=\left\{k+\sum_{i=1}^na_i10^i\,\,\middle\vert\begin{array}{c}n\in\N,\,\,a_1,\ldots,a_n\in[0,9]\\k+a_1+\cdots+a_n\geq1\end{array}\!\right\}.\end{equation}

Now, according to the definitions above, we will study the properties of each of these classes, and by studying their properties, we will reach more general properties in the $\Z/10\Z$ set. With the help of these properties, we can generalize them to the set of $\N$ numbers.

\begin{lem}\label{001}
Let $k$ be a natural number. Then:
\begin{enumerate}
\item If $x\in[2k+1]^+$, then $\Col(x)\in[6k+4]^+$.
\item If $x\in[2k]^+$, then $\Col(x)\in [k]^+\cup[k+5]^+$.
\end{enumerate}
\end{lem}
\begin{proof}
(1) For $x\in[2k+1]^{+}$, there are $n\in\N$ and $a_1,\ldots,a_n\in[0,9]$ such that\[\Col(x)=3x+1=3\left((2k+1)+\sum_{i=1}^na_i10^i\right)+1\equiv 6x+4\mod10.\](2) Similarly, for $x\in[2k]^{+}$, there are $n\in\N$ and $a_1,\ldots,a_n\in[0,9]$ such that\[\Col(x)=\frac{x}{2}=\frac{1}{2}\left(2k+\sum_{i=1}^na_i10^i\right)=k+\sum_{i=1}^n\frac{a_i}{2}10^i=k+\sum_{i=2}^n\frac{a_i}{2}10^i+5a_1.\]If $a_1$ is odd, that is, $a_1=2t+1$ for some $t\in\N$, then\[\Col(x)=k+\sum_{i=2}^n\frac{a_i}{2}10^i+10t+5\equiv 5+k\mod10.\]Now, if $a_1$ is even, that is, $a_1=2t$ for some $t\in\N$, then\[\Col(x)=k+\sum_{i=2}^n\frac{a_i}{2}10^i+10t\equiv k\mod10.\]
\end{proof}

\begin{rem}\label{005}$[k]^{+} = k$\end{rem}
The following remark will be useful to construct a graph in Section \ref{003}.
\begin{rem}\label{004}
More specifically, by Lemma \ref{001} and Remark \ref{005}, we obtain the following properties:
\[\begin{array}{llll}
\PP{0}&\Col(0)\in [5]^+\cup[0]^+\quad&\PP{5}&\Col(5)\in[6]^+\\
\PP{1}&\Col(1)\in[4]^+\quad&\PP{6}&\Col(6)\in[8]^+\cup[3]^+\\
\PP{2}&\Col(2)\in[1]^+\cup[6]^+\quad&\PP{7}&\Col(7)\in[2]^+\\
\PP{3}&\Col(3)\in[0]^+\quad&\PP{8}&\Col(8)\in[4]^+\cup[9]^+\\
\PP{4}&\Col(4)\in[2]^+\cup[7]^+\quad&\PP{9}&\Col(9)\in[8]^+
\end{array}\]
\end{rem}

\section{Graph}\label{003}



Let $G$ be the directed graph $(V,E)$ with vertices $V=[0,9]$ and the following edges\[E=\{(a,b)\mid\Col(a)\equiv b\mod10\}.\]

According to the definitions above, we obtain the following lemma that will be useful to draw the graph.
\begin{lem}\label{002}
For $x\in[0,9]$, we have:
\begin{enumerate}
\item $\deg(x)=1$ if and only if $x$ is odd.
\item $\deg(x)=2$ if and only if $x$ is even.
\end{enumerate}
\end{lem}
\begin{proof}
It is a consequence of Remark \ref{004} and Lemma \ref{001}.
\end{proof}

By using Lemma \ref{001}, Remark \ref{005} and Lemma \ref{002}, the graph $G$ can be drawn as follows:
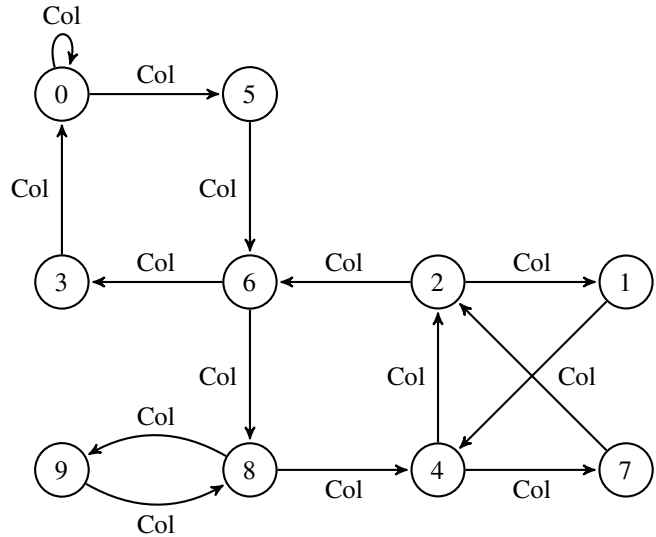
\[\begin{tikzpicture}[
    ->,>=stealth',
    auto,
    main/.style={
        circle,draw,
        font=\sffamily\Large\bfseries
    },
    node distance=2.5cm,
    shorten >=1pt,
    thick]
\node[main](0){$_0$};
\node[main](3)[below of=0]{$_3$};
\node[main](6)[right of=3]{$_6$};
\node[main](5)[right of=0]{$_5$};
\node[main](8)[below of=6]{$_8$};
\node[main](9)[left  of=8]{$_9$};
\node[main](4)[right of=8]{$_4$};
\node[main](2)[right of=6]{$_2$};
\node[main](7)[right of=4]{$_7$};
\node[main](1)[right of=2]{$_1$};
\path[every node/.style={font=\sffamily\small}]
    (0)edge node[above]{$\Col$}(5)
       edge[loop above]node{$\Col$}(0)
    (3)edge node[left]{$\Col$}(0)
    (6)edge node[above]{$\Col$}(3)
       edge node[left]{$\Col$}(8)
    (5)edge node[left]{$\Col$}(6)
	(8)edge[bend right]node[above]{$\Col$}(9)
    (9)edge[bend right]node[below]{$\Col$}(8)
	(8)edge node[below]{$\Col$}(4)
	(4)edge node[left]{$\Col$}(2)
	   edge node[below]{$\Col$}(7)
	(2)edge node[above]{$\Col$}(6)
	   edge node[above]{$\Col$}(1)	
	(7)edge node[left]{}(2)
	(1)edge node[right]{$\,\,\,\Col$}(4)
;\end{tikzpicture}\]

\section{Generalization}

The Collatz conjecture can be expressed as follows.
\begin{con}[Generalized Collatz]
For $m\in\N$, the sequence $\Col^n(m)$ ($n\geq1$) will reach $1$, if and only if, the following properties hold:
\begin{enumerate}
\item[{\tt(C1)}]$\lim\limits_{n\to\infty}\Col^n(m)\notin\{m\}$.
\item[{\tt(C2)}]$\lim\limits_{n\to\infty}\Col^n(m)\neq\infty$.
\end{enumerate}
\end{con}

Now we can claim that this conjecture is equivalent to the Collatz conjecture, because if a number does not reach itself and also does not diverge infinitely, then that number will necessarily shrink to 1. In fact, we can say that the graph we have drawn shows that if there is a loop, then what does that loop look like, and if a proof can be given for this conjecture, this graph can be used.


\begin{thebibliography}{}
\bibitem[1]{A} J.C. Lagarias. The 3x + 1 problem: An annotated bibliography (1963-1999). Sorted by author. URL: https://arxiv.org/abs/math/0309224.

\bibitem[2]{C11} J.J. O’Connor and E.F. Robertson. Lothar Collatz. https://mathshistory. st-andrews.ac.uk/Biographies/Collatz/, 2006. St Andrews University School of Mathematics and Statistics, Scotland.

\bibitem[3]{C22} D.L. Johnson and C.D. Maddux. Logo: A Retrospective. Computers in the Schools Monographs/ Separates. Taylor $\&$ Francis, 1997.

\bibitem[4]{C33} Jeffrey C. Lagarias. The 3x + 1 problem and its generalizations. The American Mathematical Monthly, 92(1):3–23, 1985

\bibitem[5]{C55} C.A. Pickover. Wonders of Numbers: Adventures in Mathematics, Mind, and Meaning. Oxford University Press, 2003.

\bibitem[6]{C77} D.R. Hofstadter. Godel, Escher, Bach: An Eternal Golden Braid. Basic Books, 1999.

\bibitem[7]{C88} R. Guy. Unsolved Problems in Number Theory. Problem Books in Mathematics. Springer New York, 2013.

\bibitem[8]{C99} Richard K. Guy. Don’t try to solve these problems. The American Mathematical
Monthly, 90(1):35–41, 1983.

\bibitem[9]{C99110} J.C. Lagarias. The Ultimate Challenge: The 3x+1 Problem. American Mathematical Society, 2010.
\end{thebibliography}
\end{document}